\documentclass[11pt, amsfonts]{amsart}

\usepackage{amsmath,amssymb,amsthm,enumerate,comment}
\usepackage[all]{xy}
\usepackage{color}
\usepackage{graphicx}
\usepackage[T1]{fontenc}
\usepackage[dvipdfm,colorlinks=true]{hyperref}
\usepackage[colorlinks=true]{hyperref}
\usepackage{mathrsfs}

\SelectTips{eu}{12}

\theoremstyle{definition}
\newtheorem{theorem}{Theorem}[section]
\newtheorem{definition}[theorem]{Definition}
\newtheorem{lemma}[theorem]{Lemma}
\newtheorem{proposition}[theorem]{Proposition}
\newtheorem{corollary}[theorem]{Corollary}

\makeatletter
\@addtoreset{equation}{section} 
\makeatother

\DeclareMathOperator{\id}{id}

\DeclareMathOperator{\Diff}{Diff}
\DeclareMathOperator{\Sing}{Sing}
\DeclareMathOperator{\vol}{vol}

\newcommand{\RR}{\mathbb{R}}

\newcommand{\ZZ}{\mathbb{Z}}

\newcommand{\GG}{\Gamma}
\renewcommand{\gg}{\gamma}
\newcommand{\omb}{\overline{\omega}}
\newcommand{\phib}{\overline{\phi}}
\newcommand{\psib}{\overline{\psi}}

\newcommand{\A}{A}

\newcommand{\CCC}{\mathrm{C}}
\newcommand{\HHH}{\mathrm{H}}

\newcommand{\rel}{\mathrm{rel}}
\newcommand{\near}{\mathrm{c}}

\newcommand{\Grel}{G_{\rel}}
\newcommand{\Gnear}{G_{\near}}

\newcommand{\BGdel}{BG^{\delta}}
\newcommand{\MGdel}{MG^{\delta}}
\newcommand{\BGbar}{\overline{BG}}

\newcommand{\BDdelM}{B\Diff_{\omega,0}(M,\partial M)^{\delta}}
\newcommand{\MDdel}{M\Diff_{\omega,0}(M,\partial M)^{\delta}}
\newcommand{\NDdel}{N\Diff_{\omega,0}^c(N)^{\delta}}
\newcommand{\BDdelN}{B\Diff_{\omega,0}^c(N)^{\delta}}
\newcommand{\BDbarN}{\overline{B\Diff_{\omega,0}^c(N)}}
\newcommand{\DM}{\Diff_{\omega,0}(M,\partial M)}
\newcommand{\DN}{\Diff_{\omega,0}^c(N)}
\newcommand{\DS}{\Diff_0(S^{n-1})}
\newcommand{\DSdel}{\Diff_0(S^{n-1})^{\delta}}
\newcommand{\DD}{\Diff_{\omega, 0}(D^n)}
\newcommand{\DDrel}{\Diff_{\omega, 0}(D^n, S^{n-1})}
\newcommand{\BDS}{B\Diff_0(S^{n-1})^{\delta}}
\newcommand{\DMful}{\Diff_{\omega,0}(M)}
\newcommand{\BDfuldelM}{B\DMful^{\delta}}
\newcommand{\MDfuldel}{M\DMful^{\delta}}
\newcommand{\Dbdr}{\Diff_0(\partial M)}
\newcommand{\BDD}{B\DD^{\delta}}

\makeatletter
\@namedef{subjclassname@2020}{%
  \textup{2020} Mathematics Subject Classification}
\makeatother

\title[McDuff's secondary class and the Euler class]{McDuff's secondary class and the Euler class of foliated sphere bundles}
\author{Shuhei Maruyama}
\address{School of Mathematics and Physics, College of Science and Engineering, Kanazawa University, Kakuma-machi, Kanazawa, Ishikawa, 920-1192, Japan}
\email{smaruyama@se.kanazawa-u.ac.jp}


\begin{document}

\begin{abstract}
  Tsuboi proved that the Calabi invariant of the closed disk transgresses to the Euler class of foliated circle bundles and suggested looking for its higher-dimensional analog.
  In this paper, we construct a cohomology class of the group of volume-preserving diffeomorphisms of a real-cohomologically acyclic manifold with sphere boundary, which is closely related to McDuff's secondary class, and prove that this cohomology class transgresses to the Euler class of foliated sphere bundles.
\end{abstract}

\maketitle

\section{Introduction}\label{sec:intro}
\begin{comment}
Let $B^n = \{ (x_1, \cdots, x_n) \in \RR^n \mid x_1^2 + \cdots x_n^2 \leq 1 \}$ be the closed unit ball in $\RR^n$ and $\omega = dx_1 \wedge \cdots \wedge dx_n$ the standard volume form on $B^n$.
Let $G = \Diff_{\omega}(B^n)_0$ be the identity component of the group of volume preserving $C^{\infty}$-diffeomorphisms of $B^n$.
Let $S^{n-1} = \partial B^n$ be the boundary sphere and $G_{\partial} = \Diff(S^{n-1})_0$ the identity component of the group of $C^{\infty}$-diffeomorphisms of $S^{n-1}$.
Then, restricting the diffeomorphisms in $G$ to the boundary $S^{n-1} = \partial B^n$ provides a map $r \colon G \to G_{\partial}$.

In section \ref{sec:cocycle}, we construct several cocycles $c_i$ of $\Grel$ which are cohomologous to each other.
Let $c \in \HHH^{n-1}(B\Grel^{\delta})$ be the cohomology class represented by $c_i$.
This class $c$ is closely related to the secondary class $s(M)$ constructed by McDuff \cite{MR0678355}.
More precisely, we prove the following.
\begin{theorem}\label{thm:sec_class_cocycle}
  Let $i \colon \Gnear \to \Grel$ be the inclusion and $j \colon \overline{B\Gnear} \to B\Gnear^{\delta}$ be the canonical map.
  Then, the pullback $j^*i^* c \in \HHH^{n-1}(\overline{B\Gnear})$ coincides with $s(M)$.
\end{theorem}
McDuff proved in \cite{MR0678355} that $s(M) \neq 0$ if $n$ is even.
Hence we obtain the following.
\begin{corollary}
  The class $c \in \HHH^{n-1}(B\Grel^{\delta})$ is non-zero if $n$ is even.
\end{corollary}
Note that the class $c$ is zero if $n$ is odd (Proposition \ref{}).
\end{comment}

Let $\DS$ be the identity component of the group of orientation-preserving $C^{\infty}$-diffeomorphisms of the sphere $S^{n-1}$.
Let $\BDS$ denote the classifying space of $\DS^{\delta}$ (for a topological group $G$, $G^{\delta}$ denotes the group $G$ with the discrete topology).
Let $e \in \HHH^{n}(\BDS)$ be the Euler class of foliated $S^{n-1}$-bundles with coefficients in $\RR$.
Let $D^n$ be the $n$-dimensional closed ball with a volume form $\omega$.
It is known that if the holonomy of a foliated $D^n$-bundle preserves the volume form $\omega$, then the Euler class of the bundle is zero (\cite{MR1183731} and \cite{misc/24370500}; see also \cite{tsuboi00}).
In other words, we have
\begin{align*}
  r^*e = 0 \in \HHH^{n}(\BDD),
\end{align*}
where $r \colon \BDD \to \BDS$ is the map induced by the restriction to the boundary.
%Let $j \colon \DDrel \to \DD$ be the inclusion.

For a fibration $F \xrightarrow{i} E \xrightarrow{p} B$, a class $x \in \HHH^{n-1}(F)$ is said to \emph{transgress to} $y \in \HHH^n(B)$ if there exist an $(n-1)$-cochain $u$ on $E$ and an $n$-cocycle $v$ representing $y$ such that
\begin{align*}
  du = p^* v \ \text{ and } \ [i^* u] = x,
\end{align*}
where $d$ is the coboundary map of the cochain complex we use (see \cite[Section 9]{MR72426}).
Note that if a class $y \in \HHH^n(B)$ satisfies $p^*y = 0$, then there exists an element $x \in \HHH^{n-1}(F)$ which transgresses to $y$.

Let $\DDrel$ be the identity component of the group of volume-preserving $C^{\infty}$-diffeomorphisms of $D^n$ which fix the boundary $S^{n-1}$ pointwise.
Considering the fibration $B\DDrel^{\delta} \to \BDD \xrightarrow{r} \BDS$ and the fact $r^* e = 0$ stated above, we can ask which class transgresses to the Euler class $e$.
For the case of the disk $D^2$ and the boundary $S^1$, Tsuboi essentially proved in \cite{tsuboi00} that the Calabi invariant, which is a non-zero element of $\HHH^1(B\Diff_{\omega,0}(D^2, S^1)^{\delta})$, transgresses to the Euler class of foliated $S^1$-bundles up to a non-zero constant multiple (see also \cite{bowden11} and \cite{Moriyoshi16}).
Tsuboi (\cite{tsuboi00}) also suggested looking for analogous results in higher-dimensional cases.

In the present paper, we consider the above problem in the following slightly generalized setting.
Let $M$ be a compact connected orientable $n$-dimensional manifold with a volume form $\omega$.
We assume that the boundary $\partial M$ is diffeomorphic to the sphere $S^{n-1}$ and that $\HHH^l(M,\partial M) = 0$ for every $0 < l < n$.
Let $\DMful$ be the identity component of the group of volume-preserving $C^{\infty}$-diffeomorphisms of $M$.
The boundary-restricting map $\DMful \to \Dbdr$ induces the fibration
\begin{align*}
  B\DM^{\delta} \xrightarrow{j} B\DMful^{\delta} \xrightarrow{r} B\Dbdr^{\delta},
\end{align*}
where $\DM$ is the identity component of the subgroup of $\DMful^{\delta}$ whose elements fix the boundary pointwise.

In Section \ref{sec:cocycle}, we introduce a cohomology class $c(M) \in \HHH^{n-1}(\BDdelM)$. %(in a slightly generalized setting).
Note that if $n=2$, the class $c(M)$ coincides with the Calabi invariant.
Hence Tsuboi's theorem asserts that the class $c(D^2)$ for the closed disk $D^2$ transgresses to the Euler class of foliated $S^1$-bundles.
The main result of this paper is the following, which provides a higher-dimensional analog of Tsuboi's theorem.
\begin{theorem}\label{thm:cM_transgression}
  The class $c(M)$ transgresses to the Euler class $e \in \HHH^n(B\DS^{\delta})$ of foliated $S^{n-1}$-bundles up to a non-zero constant multiple.
\end{theorem}

If $n$ is odd, then the real Euler class $e$ is zero.
We prove that $c(M) = 0$ if $n$ is odd (Proposition \ref{prop:cM_zero}).
Hence the theorem is trivial when $n$ is odd.

%Let $(M,\omega)$ be a compact connected orientable $n$-dimensional manifold $M$ with boundary $\partial M$ and with a volume form $\omega$.
%Let $\DMful$ (resp. $\Dbdr$) be the identity component of the group of volume-preserving $C^{\infty}$-diffeomorphisms of $M$ (resp. the identity component of the group of $C^{\infty}$-diffeomorphisms of the boundary).

%Let $\DM$ be the identity component of the group of $\omega$-preserving $C^{\infty}$-diffeomorphisms of $M$ which are identity on $\partial M$.
We also clarify a relation between the class $c(M)$ and McDuff's secondary class.
Let $N$ be the interior of $M$ equipped with the volume form which is the restriction of $\omega$ to $N$ (denoted by the same symbol $\omega$).
Let $\DN$ be the identity component of the group of compactly supported $C^{\infty}$-diffeomorphisms of $N$ which preserve the volume form $\omega$.
Let $\BDbarN$ be the homotopy fiber of $\BDdelN \to B\DN$.
Then McDuff's secondary class $s(N)$ is defined as an $(n-1)$-th cohomology class of $\BDbarN$ (see Section \ref{subsec:McDuff_class}).

Let $\iota \colon \DN \to \DM$ be the inclusion given by extending diffeomorphisms by the identity on $\partial M$.
Then we have the maps
\begin{align*}
  \HHH^{n-1}(\BDdelM) \xrightarrow{\iota^*} \HHH^{n-1}(\BDdelN) \xrightarrow{f^*} \HHH^{n-1}(\BDbarN).
\end{align*}

\begin{theorem}\label{thm:cM=sN}
  The pullback $f^*\iota^* c(M)$ coincides with McDuff's secondary class $s(N)$.
\end{theorem}
McDuff proved in \cite{MR0678355} that the class $s(N)$ is non-zero if and only if $n$ is even.
Hence this, together with Proposition \ref{prop:cM_zero}, implies the following.
\begin{corollary}\label{cor:cM_nonzero}
  The class $c(M)$ is non-zero if and only if $n$ is even.
\end{corollary}

Note that the definition of the class $c(M)$, Theorem \ref{thm:cM=sN} and Corollary \ref{cor:cM_nonzero} do not need the assumption $\partial M \cong S^{n-1}$ (see the beginning of Sections \ref{sec:cocycle} and \ref{sec:proof_cM=sN}).

%As is the case of $n = 2$, we prove the following.
%\begin{theorem}
%  If the boundary $\partial M$ is an odd dimensional sphere $S^{n-1}$, then the class $c(M)$ transgresses to $-\vol(M)^2$ times the Euler class $e \in \HHH^n(\BDS)$.
%\end{theorem}

\section{Preliminalies}
\begin{comment}
\subsection{Group cohomology}
Let $\GG$ be a group and $\A$ a right $\GG$-module.
We set $\CCC^0(\GG;\A) = \A$ and $\CCC^p(\GG; \A) = \{ c \colon \GG^p \to \A \}$ for an integer $p > 0$.
An element $c \in \CCC^p(\GG;\A)$ is called a \emph{$p$-cochain of $\GG$ with coefficients in $\A$}.
The \emph{coboundary map} $\delta \colon \CCC^p(\GG; \A) \to \CCC^{p+1}(\GG;\A)$ is defined by
\begin{align*}
  &\delta c(\gg_1, \cdots, \gg_{p+1}) \\
  &= c(\gg_2, \cdots, \gg_{p+1}) + \sum_{i=1}^{p}(-1)^{i}c(\gg_1, \cdots, \gg_i\gg_{i+1}, \cdots, \gg_{p+1}) + (-1)^{p+1}c(\gg_1, \cdots, \gg_p)^{\gg_{p+1}}.
\end{align*}
Then $(\CCC^*(\GG;\A), \delta)$ gives rise to a cochain complex, and its homology $\HHH^*(\GG;\A)$ is called the \emph{group cohomology of $\GG$ with coefficients in $\A$}.
\end{comment}

\subsection{Simplicial manifolds}\label{subsec:semisimp_mfd}
In this section, we briefly review the notion of simplicial manifold (see \cite{D} for details).

A \emph{simplicial manifold} $X_{\bullet}$ is a sequence $X_{\bullet} = \{ X_p \}_{p\geq 0}$ of manifolds endowed with smooth maps $\epsilon_i \colon X_p \to X_{p-1}$ %for $p \geq 1$ and for $0 \leq i \leq p$
and $\eta_i \colon X_p \to X_{p+1}$ for %$p \geq 0$ and
$0 \leq i \leq p$ satisfying
\begin{align*}
  &\epsilon_i \epsilon_j = \epsilon_{j-1} \epsilon_i \  \text{ for } i < j, \\
  &\eta_i \eta_j = \eta_{j+1} \eta_j \  \text{ for } i \leq j, \text{ and } \\
  &\epsilon_i \eta_j =
  \begin{cases}
    \eta_{j-1} \epsilon_i & \text{ for } i < j \\
    \id_{X_p} & \text{ for } i = j, j+1 \\
    \eta_{j} \epsilon_{i-1} & \text{ for } i > j+1.
  \end{cases}
\end{align*}
The \emph{(fat) realization $\| X_{\bullet} \|$ of a simplicial manifold $X_{\bullet}$} is the quotient space of $\coprod_p \Delta^p \times X_p$ with the identifications $(\epsilon^i(t), x) \sim (t, \epsilon_i(x))$.
Here $\Delta^p \subset \RR^{p+1}$ is the standard $p$-simplex given by
\begin{align*}
  \Delta^p = \{ (t_0, \cdots, t_p) \in \RR^{p+1} \mid t_0 + \cdots + t_p = 1, t_i \geq 0 \text{ for each } i \}
\end{align*}
and $\epsilon^i \colon \Delta^{p-1} \to \Delta^{p}$ is the $i$-th face map given by
\begin{align*}
  \epsilon^i(t_0, \cdots, t_{p-1}) = (t_0, \cdots, t_{i-1}, 0, t_i, \cdots, t_{p-1})
\end{align*}
for $0 \leq i \leq p$.
We define $\eta^i \colon \Delta^{p+1} \to \Delta^{p}$ by
\begin{align*}
  \eta^i(t_0, \cdots, t_{p+1}) = (t_0, \cdots, t_{i-1}, t_i + t_{i+1}, t_{i+2}, \cdots, t_{p+1})
\end{align*}
for $0 \leq i \leq p$.

Let $G$ be a (possibly infinite-dimensional) Lie group which acts on a manifold $M$ by $C^{\infty}$-diffeomorphisms.
These data induce four simplicial manifolds, $\BGdel_{\bullet}$, $\MGdel_{\bullet}$, $\BGbar_{\bullet}$ and $\BGbar_{\bullet}\times M$, as follows.

For $p \geq 0$, we set $\BGdel_p = (G^{\delta})^p$ and define $\epsilon_i \colon \BGdel_{p} \to \BGdel_{p-1}$ and $\eta_i \colon \BGdel_{p} \to \BGdel_{p+1}$ by
\begin{align*}
  \epsilon_i(g_1, \cdots, g_p) =
  \begin{cases}
    (g_2, \cdots, g_p) & i = 0 \\
    (g_1, \cdots, g_i g_{i+1}, \cdots, g_p) & 1 \leq i \leq p-1 \\
    (g_1, \cdots, g_{p-1}) & i = p
  \end{cases}
\end{align*}
and
\begin{align*}
  \eta_i(g_1, \cdots, g_p) = (g_1, \cdots, g_i, \id, g_{i+1}, \cdots, g_{p}),
\end{align*}
where $\id$ denotes the unit element of $G$.
Here we regard $\BGdel_p$ as a $0$-dimensional manifold.
Similarly, we set $\MGdel_p = (G^{\delta})^p \times M$ and define $\epsilon_i \colon \MGdel_{p} \to \MGdel_{p-1}$ and $\eta_i \colon \MGdel_{p} \to \MGdel_{p+1}$ by
\begin{align*}
  \epsilon_i(g_1, \cdots, g_p; x) =
  \begin{cases}
    (g_2, \cdots, g_p; x) & i = 0 \\
    (g_1, \cdots, g_i g_{i+1}, \cdots, g_p; x) & 1 \leq i \leq p-1 \\
    (g_1, \cdots, g_{p-1}; g_p(x)) & i = p
  \end{cases}
\end{align*}
and
\begin{align*}
  \eta_i(g_1, \cdots, g_p; x) = (g_1, \cdots, g_i, \id, g_{i+1}, \cdots, g_{p}; x).
\end{align*}
Then it is directly checked that $\BGdel_{\bullet} = \{ \BGdel_p \}_{p \geq 0}$ and $\MGdel_{\bullet} = \{ \MGdel_p \}_{p \geq 0}$ are simplicial manifolds.

\begin{comment}
\begin{align*}
  \eta_j \epsilon_{p} = \epsilon_{p+1}\eta_{j} \colon X_p \to X_p  (j < p )\\
  \eta_j \epsilon_{p}(g_1, \cdots, g_p;x) = \eta_j(g_1, \cdots, g_{p-1};g_p(x))  = (g_1, \cdot, \id, \cdots, g_{p-1};g_p(x))\\
  \epsilon_{p+1} \eta_j(g_1, \cdots, g_p;x) = \epsilon_{p+1} (g_1, \cdots, \id, \cdots, g_p;x)
\end{align*}
\end{comment}

For $p \geq 0$, let $\Sing^p(G)$ be the set of $C^{\infty}$ singular $p$-simplices in $G$.
We set $\BGbar_p = \Sing^p(G)/G^{\delta}$, where the group $G^{\delta}$ acts on $\Sing^p(G)$ by the multiplication on the right.
Here we regard $\BGbar_p$ as a $0$-dimensional manifold.
We define $\epsilon_i \colon \BGbar_{p} \to \BGbar_{p-1}$ and $\eta_i \colon \BGbar_{p} \to \BGbar_{p+1}$ by $\epsilon_i([\sigma]) = [\sigma \circ \epsilon^i]$ and $\eta_i([\sigma]) = [\sigma \circ \eta^i]$, respectively.
Similarly, we define $\epsilon_i \colon \BGbar_{p} \times M \to \BGbar_{p-1} \times M$ and $\eta_i \colon \BGbar_{p} \times M \to \BGbar_{p+1} \times M$ by $\epsilon_i([\sigma],x) = ([\sigma \circ \epsilon^i],x)$
and $\eta_i([\sigma],x) = ([\sigma \circ \eta_i],x)$, respectively.
Then $\BGbar = (\BGbar_p)_{p \geq 0}$ and $\BGbar_{\bullet} \times M = (\BGbar_p \times M)_{p \geq 0}$ turn out to be simplicial manifolds.

It is known that the realizations, $\| \BGdel_{\bullet} \|$, $\| \MGdel_{\bullet} \|$, $\| \BGbar_{\bullet} \|$ and $\| \BGbar_{\bullet} \times M \|$, provide models for $BG^{\delta}$, $EG^{\delta} \times_{G^{\delta}} M$, $\BGbar$ and $\BGbar \times M$,
respectively (see \cite{D}, \cite{MR505649} and \cite{Mather2011}).

\subsection{Simplicial de Rham theory}
In this section, we briefly review simplicial de Rham theory (see \cite{D} for details).

First we recall the notion of double complex.
A \emph{(first quadrant) double complex} $(A^{*,*}, d', d'')$ is a family $A^{*,*} = \{ A^{p,q} \}_{p,q \geq 0}$ of modules equipped with differentials
\begin{align*}
  d' \colon A^{p,q} \to A^{p+1, q} \ \text{ and } \ d'' \colon A^{p,q} \to A^{p,q+1}
\end{align*}
for every $p,q \geq 0$ with $d' \circ d' = 0, d'' \circ d'' = 0$ and $d' \circ d'' + d'' \circ d' = 0$.
We set
\begin{align*}
  A^n = \coprod_{p+q=n} A^{p,q}
\end{align*}
for every $n \geq 0$.
Then, the sum $d' + d''$ induces a cochain complex structure on $A^* = \{ A^n \}_{n \geq 0}$.
This cochain complex $(A^*, d)$ is called the \emph{total complex of a double complex $(A^{*,*}, d', d'')$}.
\begin{definition}\label{def:notation_decomp}
  Let $(A^{*,*}, d', d'')$ be a double complex and $(A^*,d)$ its total complex.
  Let $n$ be a non-negative integer.
  For an element $a \in A^{n}$, let $\{ a^{(l)} \}_{0 \leq l \leq n}$ be such that $a = a^{(0)} + \cdots + a^{(n)}$ and $a^{(l)} \in A^{n-l,l}$ for every $l$.
\end{definition}

Let $X_{\bullet}$ be a simplicial manifold.
For an integer $m \geq 0$, a \emph{simplicial $m$-form $\lambda$ on $X_{\bullet}$} is a family $\lambda = \{ \lambda_p \}_{p \geq 0}$ of $m$-forms $\lambda_p \in \Omega^m(\Delta^p \times X_p)$ satisfying
\begin{align*}
  (\epsilon^i \times \id)^* \lambda_p = (\id \times \epsilon_i)^* \lambda_{p-1} \in \Omega^m(\Delta^{p-1} \times X_p)
\end{align*}
for every $p \geq 0$ and $0 \leq i \leq p$.
Let $A^m(X_{\bullet})$ be the set of simplicial $m$-forms on $X_{\bullet}$.
The exterior derivative $d \colon \Omega^m(\Delta^p \times X_p) \to \Omega^{m+1}(\Delta^p \times X_p)$ induces the map $d \colon A^m(X_{\bullet}) \to A^{m+1}(X_{\bullet})$, which gives a cochain complex $(A^*(X_{\bullet}), d)$.

This cochain complex $(A^*(X_{\bullet}), d)$ can be seen as the total complex of a double complex $(A^{*,*}(X_{\bullet}), d_{\Delta}, d_X)$ defined as follows.
For non-negative integers $k$ and $l$ with $k+l = m$, let $A^{k,l}(X_{\bullet})$ be the set of simplicial $m$-forms $\lambda = \{ \lambda_p \}_{p \geq 0}$ which has degree $l$ in the $X_p$-variables for every $p \geq 0$.
The map $d_{\Delta} \colon A^{k,l}(X_{\bullet}) \to A^{k+1,l}(X_{\bullet})$ is the exterior derivative with respect to the variables in the standard simplex $\Delta^p$ and the map $d_X \colon A^{k,l}(X_{\bullet}) \to A^{k,l+1}(X_{\bullet})$ is $(-1)^k$ times the exterior derivative with respect to the $X_p$-variables on each $\Delta^p \times X_p$.

There is another double complex $(\mathscr{A_{\bullet}}^{*,*}(X), \delta, d'')$ defined as follows.
For a simplicial manifold $X_{\bullet}$, we set $\mathscr{A}^{k,l}(X_{\bullet}) = \Omega^l(X_k)$.
Let $\delta \colon \mathscr{A}^{k,l}(X_{\bullet}) \to \mathscr{A}^{k+1,l}(X_{\bullet})$ be the map defined by
\begin{align*}
  \delta = \sum_{i = 0}^{k+1} (-1)^i \epsilon_i^*
\end{align*}
and $d'' \colon \mathscr{A}^{k,l}(X_{\bullet}) \to \mathscr{A}^{k,l+1}(X_{\bullet})$ be $(-1)^k$ times the exterior derivative of the de Rham complex $\Omega^*(X_k)$.
Then $(\mathscr{A}^{*,*}(X_{\bullet}), \delta, d'')$ gives rise to a double complex.
Let $(\mathscr{A}^*(X_{\bullet}), d)$ denote its total complex.

\begin{theorem}[\cite{MR413122}]\label{thm:simplicial_derham_thm}
  There are isomorphisms
  \begin{align*}
    \HHH^*(\mathscr{A}^*(X_{\bullet})) \to \HHH^*(A^*(X_{\bullet})) \to \HHH^*(\| X_{\bullet} \|).
  \end{align*}
\end{theorem}

An isomorphism $\HHH^*(\mathscr{A}^*(X_{\bullet})) \to \HHH^*(A^*(X_{\bullet}))$ is given as follows.
For a sequence $I = (i_0, \cdots, i_k)$ of integers $0 \leq i_0 < \cdots < i_k \leq p$, we set
\begin{align*}
  %\mu^{I} = \epsilon^{j_{p-k}} \cdots \epsilon^{j_1} \colon \Delta^{k} \to \Delta^{p} \\
  \mu_{I} = \epsilon_{j_1} \cdots \epsilon_{j_{p-k}} \colon X_p \to X_k,
\end{align*}
where $0 \leq j_1 < \cdots < j_{p-k} \leq p$ is the complementary sequence of $I$.
For $\alpha \in \mathscr{A}^{k,l}(X_{\bullet})$, we define a simplicial $(k+l)$-form $\mathscr{E}(\alpha) = \{ \mathscr{E}(\alpha)_p \}_{p \geq 0}$ by
\begin{align*}
  \mathscr{E}(\alpha)_p =
  \begin{cases}
    0 & \text{ if } p < k \\
    \displaystyle k! \sum_{|I| = k}\left(\sum_{j=0}^k(-1)^j t_{i_j}dt_{i_0} \wedge \cdots \wedge \widehat{dt}_{i_j} \wedge \cdots \wedge dt_{i_k}\right)\wedge \pi_p^* \mu_I^* \alpha & \text{ if } p \geq k,
  \end{cases}
\end{align*}
where the hat means that the term under the hat is omitted and $\pi_p \colon \Delta^p \times X_p \to X_p$ is the projection.
This defines a map $\mathscr{E} \colon \mathscr{A}^*(X_{\bullet}) \to A^*(X_{\bullet})$, which induces an isomorphism $\HHH^*(\mathscr{A}^*(X_{\bullet})) \to \HHH^*(A^*(X_{\bullet}))$.

\subsection{McDuff's secondary class}\label{subsec:McDuff_class}
In this section we briefly explain the definition and properties of the secondary class $s(N)$ constructed by McDuff \cite{MR0678355}.

Let $N$ be a connected orientable $n$-dimensional $C^{\infty}$-manifold without boundary and with a volume form $\omega$.
We assume that $N$ is noncompact and that $\HHH_c^i(N) = 0$ for all $i < n$.
Let $\Diff_{\omega,0}^c(N)$ be the identity component of the group of compactly supported $C^{\infty}$-diffeomorphisms of $N$ which preserve $\omega$.

\begin{comment}
In this section, we consider the following two cases simultaneously:
\begin{itemize}
  \item[\textbf{Case 1.}] $M$ is a noncompact manifold without boundary such that $\HHH_c^i(M) = 0$ for all $i < n$.
  In this case, we let $G$ be the identity component $\Diff_{\omega,0}^c(M)$ of the group of compactly supported $C^{\infty}$-diffeomorphisms of $M$ which preserve $\omega$.
  \item[\textbf{Case 2.}] $M$ is a compact manifold with boundary such that $\HHH^i(M, \partial M) = 0$ for all $i < n$.
  In this case, we let $G$ be the identity component $\Diff_{\omega,0}(M,\partial M)$ of the group of $C^{\infty}$-diffeomorphisms of $M$ which are identity on the boundary $\partial M$ and preserve $\omega$.
\end{itemize}
\end{comment}
%We assume that $M$ is either a noncompact manifold without boundary such that $\HHH_c^i(M) = 0$ for all $i < n$, or a compact manifold with boundary such that $\HHH^i(M, \partial M) = 0$ for all $i < n$\footnote{In \cite{MR0678355}, only the former case is treated. However, the same argument provides a well-defined cohomology class as well.}.

%Every simplicial $m$-form $\alpha \in A^{m}(\BDbarN_{\bullet} \times N)$ has a unique decomposition
%\begin{align*}
%  \alpha = \alpha^{(0)} + \cdots + \alpha^{(m)},
%\end{align*}
%where $\alpha^{(l)} \in A^{k,l}(\BDbarN_{\bullet} \times N)$.
%An element $\alpha^{(l)} \in A^{k,l}(\BDbarN_{\bullet} \times N)$ is said to be \emph{compactly supported} if the restriction of $\alpha^{(l)}$ to $\Delta^p \times \{ S \} \times N$ has compact support for every $p \geq 0$ and for every $S \in \BDbarN_{p}$,
%and
Recall the notation from Definition \ref{def:notation_decomp}.
A simplicial $m$-form $\alpha = \alpha^{(0)} + \cdots + \alpha^{(m)}$ on $\BDbarN_{\bullet} \times N$ is said to be \emph{compactly supported} if the restriction of each $\alpha^{(l)}$ to $\Delta^p \times \{ S \} \times N$ has compact support for every $p \geq 0$ and for every $S \in \BDbarN_{p}$.

%In case 2, an element $\alpha^{(l)} \in A^{k,l}(\BDbarN_{\bullet} \times N)$ is said to be \emph{relative to $\partial M$}
%if the restriction of $\alpha^{(l)}$ to $\Delta^p \times \{ S \} \times M$ is contained in $\Omega^{k+l}(\Delta^p \times \{ S \} \times M, \Delta^p \times \{ S \} \times \partial M)$ for every $p \geq 0$ and for every $S \in \BDbarN_{p}$.
%A simplicial $m$-form $\alpha$ is said to be \emph{relative to $\partial M$} if each $\alpha^{(l)}$ is relative to $\partial M$.

%Let $\BDbarN_{\bullet}$ and $\BDbarN_{\bullet} \times N$ be the semisimplicial manifolds.
For a $p$-simplex $S \in \BDbarN_p$, we take a representative $\sigma \in \Sing^p(\Diff_{\omega,0}^c(N))$ of $S$.
Let $f_{\sigma} \colon \Delta^p \times \{ S \} \times N \to N$ be the map defined by
\begin{align*}
  f_{\sigma}(t,x) = \sigma(t)^{-1}(x).
\end{align*}
%Since $f_{\sigma g} = g^{-1} f_{\sigma}$, we have $\Omega(\sigma g) = f_{\sigma g}^* \omega = f_{\sigma}^* (g^{-1})^* \omega = f_{\sigma}^* \omega = \Omega(\sigma)$.
Note that the pullback $f_{\sigma}^* \omega \in \Omega^n(\Delta^p \times \{ S \} \times N)$ does not depend on the choice of the representatives of $S$.
Hence we set $\Omega(S) = f_{\sigma}^* \omega$.
Then $\Omega(S)$ induces an $n$-form $\Omega_p \in \Omega^n(\Delta^p \times \BDbarN_p \times N)$, and these $\Omega_p$'s give rise to a simplicial $n$-form $\Omega = \{ \Omega_p \}_{p \geq 0} \in A^n(\BDbarN_{\bullet} \times N)$.
Note that under the decomposition $\Omega = \Omega^{(0)} + \cdots + \Omega^{(n)}$, the simplicial $l$-form $\Omega^{(l)}$ is compactly supported for $l < n$.

McDuff proved in \cite{MR0678355} that there exists a simplicial $(n-1)$-form $\Phi = \Phi^{(0)} + \cdots + \Phi^{(n-1)}$ on $\BDbarN_{\bullet} \times N$ such that $d \Phi = \Omega$ and that $\Phi^{(0)}, \cdots , \Phi^{(n-2)}$ are compactly supported.
In particular, the wedge product $\Phi \wedge \Omega$ is compactly supported.
The simplicial $(2n-1)$-form $\Phi \wedge \Omega$ is closed since $\omega \wedge \omega = 0$.
Hence, the integration along the fiber $N$ provides a simplicial $(n-1)$-form $\int_N(\Phi \wedge \Omega)$ on $\BDbarN_{\bullet}$, which is a closed form.
We set
\begin{align*}
  s(N) = \left[ \int_N(\Phi \wedge \Omega) \right] \in \HHH^{n-1}(\BDbarN).
\end{align*}

If $n$ is odd, then $s(N) = 0$ since $\Phi \wedge \Omega = d(\frac{1}{2}\Phi \wedge \Phi)$.
McDuff proved the following.
\begin{theorem}[\cite{MR0678355}]
  The class $s(N)$ is non-zero if $n$ is even.
\end{theorem}

\section{A closed form in $\mathscr{A}^{n-1}(\| \BDdelM_{\bullet} \|)$}\label{sec:cocycle}
Let $M$ be a compact connected orientable $n$-dimensional manifold with boundary $\partial M$ and with a volume form $\omega$.
In this section and the next section, we do \emph{not} assume that the boundary is diffeomorphic to the sphere $S^{n-1}$.
We assume that $\HHH^l(M, \partial M) = 0$ for all $l < n$.
Let $N$ be the interior of $M$.
Since $\HHH_c^*(N) \cong \HHH^*(M, \partial M)$, the interior $N$ satisfies the assumption in Section \ref{subsec:McDuff_class}.
In this section, we construct a closed form in $\mathscr{A}^{n-1}(\| \BDdelM_{\bullet} \|)$ which correspond to the class $s(N)$ under the maps
\begin{align*}
  \HHH^{n-1}(\| \BDdelM_{\bullet} \|) \to \HHH^{n-1}(\| \BDdelN_{\bullet} \|) \to \HHH^{n-1}(\| \BDbarN_{\bullet} \|).
\end{align*}

Let $\Omega^p(M, \partial M)$ be the set of $p$-forms on $M$ whose pullback to $\partial M$ is $0$.
This gives a subcomplex of the de Rham complex $(\Omega^*(M), d)$, the cohomology of which is isomorphic $\HHH^*(M, \partial M)$.

By definition, every element of $\mathscr{A}^{k,l}(\| \MDdel_{\bullet} \|)$ is identified with a map
\begin{align*}
  c \colon (\DM)^k \to \Omega^l(M).
\end{align*}
Under this identification, the coboundary map
\begin{align*}
  \delta \colon \mathscr{A}^{k,l}(\| \MDdel_{\bullet} \|) \to \mathscr{A}^{k+1,l}(\| \MDdel_{\bullet} \|)
\end{align*}
is given as
\begin{align*}
  \delta &c (g_1, \cdots, g_{k+1}) \\
  &= c(g_2, \cdots, g_{k+1}) + \sum_{i=1}^{k} (-1)^{i} c(g_1, \cdots, g_ig_{i+1}, \cdots, g_{k+1}) + (-1)^{k+1} c(g_1, \cdots, g_k)^{g_{k+1}},
\end{align*}
where $c(g_1, \cdots, g_k)^{g_{k+1}}$ denotes the pullback of $c(g_1, \cdots, g_k) \in \Omega^l(M)$ by the diffeomorphism $g_{k+1}$.

Since each $\BDdelM_k$ is a $0$-dimensional manifold, we have
\begin{align*}
  \mathscr{A}^{k,0}(\BDdelM_{\bullet}) = \{ c \colon (\DM)^{k} \to \RR \}
\end{align*}
and $\mathscr{A}^{k,l}(\BDdelM_{\bullet}) = 0$ for $l > 0$.

Let $\omega^{(n)}$ be an element of $\mathscr{A}^{0,n}(\MDdel_{\bullet})$ corresponding to the volume form $\omega$ of $M$ under the identification $\mathscr{A}^{0,n}(\MDdel_{\bullet}) = \Omega^n(M)$.
By the exactness of $\omega$, we take an $(n-1)$-form $\phi$ such that $d\phi = \omega$.
We set $\phi^{(n-1)}$ to be an element of $\mathscr{A}^{0,n-1}(\MDdel_{\bullet})$ corresponding to $\phi$.
By the assumption $\HHH^{l}(M,\partial M) = 0$ for all $l < n$ and the zig-zag argument on the double complex $\mathscr{A}^{*,*}(\MDdel_{\bullet})$, we obtain $\phi^{(l)} \in \mathscr{A}^{n-l, l}(\MDdel_{\bullet})$ such that
$\phi = \phi^{(0)} + \cdots + \phi^{(n-1)} \in \mathscr{A}^{n-1}(\MDdel_{\bullet})$ satisfies $d\phi = \omega^{(n)}$ and
$\phi^{(l)}(g_1, \cdots, g_{n-1-l}) \in \Omega^{l}(M, \partial M)$ for every $0 \leq l \leq n-2$ and for every $(g_1, \cdots, g_{n-1-l}) \in \DM^{n-1-l}$.
This is depicted as follows.
\begin{align*}
\xymatrix{
\omega^{(n)} &&&&& \\
\phi^{(n-1)} \ar[u]_-{d''} \ar[r]^-{\delta} &0&&&& \\
& \phi^{(n-2)} \ar[u]_-{d''} \ar[r]^-{\delta} &\ar@{.}[rrd]&&& \\
&&&&& \\
&&&& \phi^{(0)} \ar[u]_-{d''} \ar[r]^-{\delta} & 0
}
\end{align*}
Note that $\delta \phi^{(0)} = 0$ since $d''\delta \phi^{(0)} = -\delta d''\phi^{(0)} = -\delta \delta \phi^{(1)} = 0$ and the restriction of $\delta \phi^{(0)} (g_1, \cdots, g_n)$ on $\partial M$ is $0$.

For non-negative integers $p,q,r$ and $s$, the \emph{wedge product}
\begin{align*}
  \wedge \colon \mathscr{A}^{p,q}(\MDdel_{\bullet}) \times \mathscr{A}^{r,s}(\MDdel_{\bullet}) \to \mathscr{A}^{p+r, q+s}(\MDdel_{\bullet})
\end{align*}
is defined by
\begin{align*}
  \alpha \wedge \beta (g_1, \cdots, g_{p+r}) = \alpha(g_1, \cdots, g_{p})^{g_{p+1}\cdots g_{p+r}} \wedge \beta(g_{p+1}, \cdots, g_{p+r}),
\end{align*}
for $\alpha \in \mathscr{A}^{p,q}(\MDdel_{\bullet})$, $\beta \in \mathscr{A}^{r,s}(\MDdel_{\bullet})$ and $(g_1, \cdots, g_{p+r}) \in \DM^{p+r}$.
It is straightforward to check that the equalities
\begin{align}\label{product_differential}
  \delta (\alpha \wedge \beta) = \delta \alpha \wedge \beta + (-1)^p \alpha \wedge \delta \beta \nonumber \\
  d'' (\alpha \wedge \beta) = d'' \alpha \wedge \beta + (-1)^q \alpha \wedge d'' \beta
\end{align}
hold for $\alpha \in \mathscr{A}^{p,q}(\MDdel_{\bullet})$ and $\beta \in \mathscr{A}^{r,s}(\MDdel_{\bullet})$.
By extending the wedge product linearly, we obtain the wedge product on the total complex.

The wedge product $\phi \wedge \omega^{(n)} \in \mathscr{A}^{2n-1}(\MDdel_{\bullet})$ is a closed form since $d(\phi \wedge \omega^{(n)}) = \omega^{(n)} \wedge \omega^{(n)} =0$.
Since $\Omega^l(M) = 0$ if $l > n$, we have $\phi \wedge \omega^{(n)} = \phi^{(0)} \wedge \omega^{(n)} \in \mathscr{A}^{n-1, n}(\MDdel_{\bullet})$.
Then the integration along the fiber $M$ provides a closed $(n-1)$-form $\int_M \phi^{(0)} \wedge \omega^{(n)}$ on $\BDdelM_{\bullet}$.
We set
\begin{align*}
  c(M) = \left[ \int_M \phi^{(0)} \wedge \omega^{(n)} \right] \in \HHH^{n-1}(\BDdelM).
\end{align*}

\begin{lemma}
  The class $c(M)$ does not depend on the choice of $\phi$.
\end{lemma}

\begin{proof}
  Let $\psi = \psi^{(0)} + \cdots + \psi^{(n-1)} \in \mathscr{A}^{n-1}(\MDdel_{\bullet})$ be another choice of $\phi$.
  Then there exists $\zeta = \zeta^{(0)} + \cdots + \zeta^{(n-2)} \in \mathscr{A}^{n-2}(\MDdel_{\bullet})$ such that $\phi - \psi = d\zeta$ and $\zeta^{(l)}(g_1, \cdots, g_{n-l-2}) \in \Omega^{l}(M, \partial M)$.
  Hence we have
  \begin{align*}
    \phi^{(0)} \wedge \omega^{(n)} - \psi^{(0)} \wedge \omega^{(n)} = \delta \zeta^{(0)} \wedge \omega^{(n)} = \delta (\zeta^{(0)} \wedge \omega^{(n)}),
  \end{align*}
  which implies
  \begin{align*}
    \int_M \phi^{(0)} \wedge \omega^{(n)} - \int_M \psi^{(0)} \wedge \omega^{(n)} = \delta \left(\int_M \zeta^{(0)} \wedge \omega^{(n)} \right).
  \end{align*}
\end{proof}

\begin{proposition}\label{prop:cM_zero}
  If $n$ is odd, then the class $c(M)$ is zero.
\end{proposition}

\begin{proof}
  By \eqref{product_differential} and the assumption, we have
  \begin{align*}
    d(\phi \wedge \phi) &= d\phi \wedge \phi + \phi \wedge d\phi - 2\left(\sum_{i=1}^{(n-1)/2} \phi^{(2i-1)}\right) \wedge d\phi \\
    &= \omega^{(n)} \wedge \phi + \phi \wedge \omega^{(n)} - 2\left(\sum_{i=1}^{(n-1)/2} \phi^{(2i-1)}\right) \wedge \omega^{(n)} \\
    &= 2\phi^{(0)} \wedge \omega^{(n)}.
  \end{align*}
  %Since $d(\phi \wedge \phi) = 2\phi \wedge d\phi = 2\phi^{(0)} \wedge \omega^{(n)}$,
  Hence we obtain
  \begin{align*}
    \int_M \phi^{(0)} \wedge \omega^{(n)} = \int_M d(\phi \wedge \phi) = \int_M (\delta(\phi \wedge \phi) + d''(\phi \wedge \phi)) = \delta \left(\int_M \phi \wedge \phi \right),
  \end{align*}
  where the last equality comes from the stokes formula and $\phi^{(l)} \in \Omega^l(M, \partial M)$ for every $l < n-1$.
\end{proof}

\section{Proof of Theorem \ref{thm:cM=sN}}\label{sec:proof_cM=sN}

Let $M$ and $\omega$ be as in Section \ref{sec:cocycle}.
Let $\phi$ be an element of $\mathscr{A}^{n-1}(\NDdel_{\bullet})$ appeared in the definition of $c(M)$.
Let $N$ be the interior of $M$.
By abuse of notation, we use the same symbols $\omega, \omega^{(n)}$ and $\phi^{(l)}$ to denote the pullbacks to $N$.
%$N$ which is obtained by the restriction of the volume form of $M$.
Let $\iota \colon \DN \to \DM$ be the inclusion given by extending diffeomorphisms by the identity on $\partial M$.
%This induces a simplicial map $\iota_{\bullet} \colon \BDdelN_{\bullet} \to \BDdelM_{\bullet}$, and hence a map
This induces a map
\begin{align*}
  \iota^* \colon \HHH^{n-1}(\BDdelM) \to \HHH^{n-1}(\BDdelN).
\end{align*}

%Let $\BDbarN_{\bullet}\times N \to \BDbarN_{\bullet}$ and $\NDdel_{\bullet} \to \BDdelN_{\bullet}$ be the projections.
Let $\{ e_0, \cdots, e_p \} \subset \Delta^p$ be the canonical basis of $\RR^{p+1}$.
We define a map $f \colon \| \BDbarN_{\bullet} \| \to \| \BDdelN_{\bullet} \|$ by
\begin{align*}
  f(t, [\sigma]) = (t, \sigma(e_0)\sigma(e_1)^{-1}, \cdots, \sigma(e_{p-1})\sigma(e_p)^{-1})
\end{align*}
on each $\Delta^p \times \BDbarN_p$, and define $F \colon \| \BDbarN_{\bullet} \times N \| \to \| \NDdel_{\bullet} \|$ by
\begin{align*}
  F(t, [\sigma], x) = (t, \sigma(e_0)\sigma(e_1)^{-1}, \cdots, \sigma(e_{p-1})\sigma(e_p)^{-1}; \sigma(e_p)\sigma(t)^{-1}(x))
\end{align*}
on each $\Delta^p \times \BDbarN_p \times N$.
These maps form a commutative diagram
\begin{align*}
\xymatrix{
\| \BDbarN_{\bullet} \times N \| \ar[r]^-{F} \ar[d] & \| \NDdel_{\bullet} \| \ar[d] \\
\| \BDbarN_{\bullet} \| \ar[r]^-{f} & \| \BDdelN_{\bullet} \|.
}
\end{align*}

%\begin{lemma}
%  The pullback $F^*\mathscr{E}(\omega^{(n)}) \in A^n(\BDbarN_{\bullet} \times N)$ coincides with $\Omega$.
%\end{lemma}

Since $\omega^{(n)}$ is contained in $\mathscr{A}^{0,n}(\NDdel_{\bullet})$, the simplicial $n$-form $\mathscr{E}(\omega^{(n)}) = \{ \mathscr{E}(\omega^{(n)})_p \}_{p \geq 0}$ is given by
\begin{align*}
  \mathscr{E}(\omega^{(n)})_p = \pi_p^*\omega^{(n)} \in \Omega^n(\Delta^p \times \NDdel_p),
\end{align*}
where $\pi_p \colon \Delta^p \times \NDdel_p \to \NDdel_p$ is the projection.
  %Note that the map $\pi_p \circ F$ on $\Delta^p \times \BDbarN_p \times N$ is given by
  %\begin{align*}
  %  \pi_p \circ F(t,[\sigma],x) = (f([\sigma]), \sigma(e_{p})\sigma(t)^{-1}(x)).
  %\end{align*}
Hence, by the definition of the simplicial $n$-form $\Omega \in A^n(\BDbarN_{\bullet} \times N)$ given in Section \ref{subsec:McDuff_class}, we have $F^*\mathscr{E}(\omega^{(n)}) = \Omega$.

Let us consider the simplicial $(n-1)$-form $F^*\mathscr{E}(\phi) \in A^{n-1}(\BDbarN_{\bullet} \times N)$. %, where $\phi = \phi^{(0)} + \cdots + \phi^{(n-1)} \in \mathscr{A}^{n-1}(\NDdel_{\bullet})$ is the form satisfying $d\phi = \omega^{(n)}$.
Since $d\phi = \omega^{(n)}$, we have $dF^*\mathscr{E}(\phi) = \Omega$.
Hence $F^*\mathscr{E}(\phi)$ plays the role of $\Phi$ appeared in the definition of McDuff's secondary class $s(N)$.
We set $\Phi = F^*\mathscr{E}(\phi)$.
%\end{proof}
By the definition of $\mathscr{E}$, we have
$\mathscr{E}(\phi \wedge \omega^{(n)}) = \mathscr{E}(\phi^{(0)}\wedge \omega^{(n)}) = \mathscr{E}(\phi^{(0)}) \wedge \mathscr{E}(\omega^{(n)}) = \mathscr{E}(\phi) \wedge \mathscr{E}(\omega^{(n)})$, and hence
\begin{align*}
  F^*\mathscr{E}(\phi \wedge \omega^{(n)}) = F^* \mathscr{E} (\phi) \wedge F^* \mathscr{E}(\omega^{(n)}) = \Phi \wedge \Omega.
\end{align*}

\begin{lemma}\label{lem:commutativity_int_N}
  The integration $\int_N$ along $N$ induces a commutative diagram
  \begin{align*}
  \xymatrix{
  \mathscr{A}^{n-1,n}(\NDdel_{\bullet}) \ar[r]^-{\mathscr{E}} \ar[d]^-{\int_N} & A^{n-1,n}(\NDdel_{\bullet}) \ar[r]^-{F^*} \ar[d]^-{\int_N} & A^{n-1,n}(\BDbarN_{\bullet} \times N) \ar[d]^-{\int_N} \\
  \mathscr{A}^{n-1,0}(\BDdelN_{\bullet}) \ar[r]^-{\mathscr{E}} & A^{n-1,0}(\BDdelN_{\bullet}) \ar[r]^-{f^*} & A^{n-1,0}(\BDbarN_{\bullet}).
  }
  \end{align*}
\end{lemma}

\begin{proof}
  Let $\alpha$ be an element of $\mathscr{A}^{n-1,n}(\NDdel_{\bullet})$.
  Then $\int_N \alpha \in \mathscr{A}^{n-1,0}(\BDdelN_{\bullet})$ is given by
  \begin{align*}
    \left( \int_N \alpha \right) (g_1, \cdots, g_{n-1}) = \int_N \alpha (g_1, \cdots, g_{n-1}),
  \end{align*}
  and hence $\mathscr{E}\left( \int_N \alpha \right) = \left\{ \mathscr{E}\left( \int_N \alpha \right)_p \right\}_{p \geq 0} \in A^{n-1,0}(\BDdelN_{\bullet})$ is of the form
  \begin{align*}
    \mathscr{E}\left( \int_N \alpha \right)_p =
    \begin{cases}
      0 & \text{ if } p < n-1 \\
      \displaystyle (n-1)! \sum_{|I| = n-1}\left(\sum_{j=0}^{n-1}(-1)^j t_{i_j}dt_{i_0} \wedge \cdots \wedge \widehat{dt}_{i_j} \wedge \cdots \wedge dt_{i_{n-1}}\right)\wedge \pi_p^* \mu_I^* \left( \int_N\alpha \right) & \text{ if } p \geq n-1,
    \end{cases}
  \end{align*}
  where $\pi_p \colon \Delta^p \times \BDdelN_p \to \BDdelN_p$ is the projection.

  By definition, the element $\mathscr{E}(\alpha) = \{ \mathscr{E}(\alpha)_p \}_{p \geq 0}$ of $A^{n-1,n}(\NDdel_{\bullet})$ is given by
  \begin{align*}
    \mathscr{E}\left( \alpha \right)_p =
    \begin{cases}
      0 & \text{ if } p < n-1 \\
      \displaystyle (n-1)! \sum_{|I| = n-1}\left(\sum_{j=0}^{n-1}(-1)^j t_{i_j}dt_{i_0} \wedge \cdots \wedge \widehat{dt}_{i_j} \wedge \cdots \wedge dt_{i_{n-1}}\right)\wedge (\pi_p')^* \mu_I^* \alpha & \text{ if } p \geq n-1,
    \end{cases}
  \end{align*}
  where $\pi_p' \colon \Delta^p \times \NDdel_p \to \NDdel_p$ is the projection.
  %Note that $\mu_I \pi_p'$ is a bundle map which covers $\mu_I \pi_p$.
  Hence the integration $\int_N \mathscr{E}(\alpha) = \left\{ \int_N (\mathscr{E}(\alpha)_p) \right\}_{p \geq 0} \in A^{n-1,0}(\BDdelN_{\bullet})$ along the fiber $N$ is of the form
  \begin{align*}
    \int_N (\mathscr{E}\left( \alpha \right)_p) =
    \begin{cases}
      0 & \text{ if } p < n-1 \\
      \displaystyle (n-1)! \sum_{|I| = n-1}\left(\sum_{j=0}^{n-1}(-1)^j t_{i_j}dt_{i_0} \wedge \cdots \wedge \widehat{dt}_{i_j} \wedge \cdots \wedge dt_{i_{n-1}}\right)\wedge \pi_p^* \mu_I^* \left(\int_N \alpha \right) & \text{ if } p \geq n-1.
    \end{cases}
  \end{align*}
  This implies the commutativity of the left-hand side of the diagram.

  Let $\lambda = \{ \lambda_p \}_{p \geq 0}$ be an element of $A^{n-1,n}(\BDdelN_{\bullet})$.
  By definition, the maps $F$ and $f$ provide a bundle map
  \begin{align*}
  \xymatrix{
    \Delta^p \times \BDbarN_p \times N \ar[r]^-{F} \ar[d] & \Delta^p \times \BDdelN_p \times N \ar[d] \\
    \Delta^p \times \BDbarN_p \ar[r]^-{f} & \Delta^p \times \BDdelN_p
  }
  \end{align*}
  for each $p \geq 0$.
  Hence, by the property of the integration along the fiber, we have
  \begin{align*}
    \left(\int_N F^* \lambda \right)_p = \int_N F^* (\lambda_p) = f^* \int_N \lambda_p = \left(f^* \int_N \lambda \right)_p.
  \end{align*}
  This implies the commutativity of the right-hand side of the diagram.
\end{proof}

\begin{proof}[Proof of Theorem \ref{thm:cM=sN}]
  By Lemma \ref{lem:commutativity_int_N}, we have $f^*\mathscr{E} (\int_N \phi\wedge \omega^{(n)}) = \int_N F^* \mathscr{E}(\phi\wedge \omega^{(n)}) = \int_N \Phi \wedge \Omega$.
  Since the left-hand side represents the class $f^* \iota^* (c(M))$ and the right-hand side represents the class $s(N)$, we have $f^* \iota^* (c(M)) = s(N)$.
\end{proof}

\section{Proof of Theorem \ref{thm:cM_transgression}}\label{sec:proof_transgression}
Let $M$ and $\omega$ be as in Section \ref{sec:proof_cM=sN}.
In this section, we \emph{do} assume that the boundary $\partial M$ is diffeomorphic to the sphere $S^{n-1}$.
%We further assume that $n$ is even and the boundary $\partial M$ is an odd-dimensional sphere $S^{n-1}$.
By the cohomology long exact sequence of the pair $(M, \partial M)$ and the assumption that $\HHH^l(M, \partial M) = 0$ for every $l < n$, we have $\HHH^l(M) = 0$ for every $0 < l < n$.

Recall that $\DMful$ denotes the identity component of the group of $\omega$-preserving $C^{\infty}$-diffeomorphisms of $M$.
Let $\omb^{(n)}$ be the element of $\mathscr{A}^{0,n}(\MDfuldel_{\bullet})$ which corresponds to $\omega \in \Omega^n(M)$ under the canonical identification $\mathscr{A}^{0,n}(\MDfuldel_{\bullet}) = \Omega^n(M)$.
Let $i \colon \partial M \to M$ be the inclusion.
\begin{lemma}\label{lem:phib_boundary_coincide}
  There exists $\phib = \phib^{(0)} + \cdots + \phib^{(n-1)} \in \mathscr{A}^{n-1}(\MDfuldel_{\bullet})$ satisfying the following:
  \begin{enumerate}[$(1)$]
    \item $d \phib = \omb^{(n)} + \delta \phib^{(0)}$,
    \item $\phib^{(l)}(g_1, \cdots, g_{n-l-1}) \in \Omega^l(M, \partial M)$ for every $l < n-1$ and for every $(g_1, \cdots, g_{n-l-1}) \in \DM^{n-l-1}$, and
    \item for every $l < n-1$, the pullback $i^*\phib^{(l)}(g_1, \cdots, g_{n-l-1}) \in \Omega^l(\partial M)$ coincides with $i^*\phib^{(l)}(h_1, \cdots, h_{n-l-1})$ whenever $g_k, h_k \in \DMful$ satisfy $g_k|_{\partial M} = h_k|_{\partial M}$ for every $1 \leq k \leq n-l-1$.
  \end{enumerate}
\end{lemma}
%By the zig-zag argument on the double complex $\mathscr{A}^{*,*}(\MDfuldel_{\bullet})$, we obtain $\phib = \phib^{(0)} + \cdots + \phib^{(n-1)}$ such that
%\begin{align}\label{align:potential}
%  d \phib = \omb^{(n)} + \delta \phib^{(0)}
%\end{align}
%and $\phib^{(l)}(g_1, \cdots, g_{n-l-1}) \in \Omega^l(M, \partial M)$ for every $l < n-1$ and for every $(g_1, \cdots, g_{n-l-1}) \in \DM^{n-l-1}$.
On the double complex $\mathscr{A}^{*,*}(\MDfuldel_{\bullet})$, the condition (1) in Lemma \ref{lem:phib_boundary_coincide} is depicted as follows.
\begin{align*}
\xymatrix{
\omb^{(n)} &&&&& \\
\phib^{(n-1)} \ar[u]_-{d''} \ar[r]^-{\delta} &0&&&& \\
& \phib^{(n-2)} \ar[u]_-{d''} \ar[r]^-{\delta} &\ar@{.}[rrd]&&& \\
&&&&& \\
&&&& \phib^{(0)} \ar[u]_-{d''} \ar[r]^-{\delta} & \delta \phib^{(0)}.
}
\end{align*}
%Note that we can take $\phib^{(l)}$ as a normalized cochain.
%Since $d'' \delta \phib^{(0)} = 0$, there exists a unique cocycle $\chi \in \mathscr{A}^n(\BDfuldelM_{\bullet})$ such that $\pi^* \chi = \delta \phib^{(0)}$,
%where $\pi^* \colon \mathscr{A}^n(\BDfuldelM_{\bullet}) \to \mathscr{A}^n(\MDfuldel_{\bullet})$ is the map induced by the canonical projection $\MDfuldel_{\bullet} \to \BDfuldelM_{\bullet}$.

%Let $i \colon \partial M \to M$ be the inclusion.
%\begin{lemma}
%  The above $\{ \phib^{(l)} \}_{0 \leq l \leq n-1}$ can be taken as
%  \begin{align}\label{align:boundary_coincide}
%    i^*\phib^{(l)}(g_1, \cdots, g_{n-l-1}) = i^*\phib^{(l)}(h_1, \cdots, h_{n-l-1}) \in \Omega^l(\partial M)
%  \end{align}
%  whenever $g_i, h_i \in \DMful$ satisfy $g_i|_{\partial M} = h_i|_{\partial M}$ for every $i$.
%\end{lemma}

\begin{proof}[Proof of Lemma \ref{lem:phib_boundary_coincide}]
  %By the diagram chasing argument and $\HHH^l(M,\partial M) = \HHH^l(M) = 0$ for $1 \leq l \leq n-1$, we take $\psib = \psib^{(0)} + \psib^{(n-1)}$ satisfying (1) and (2).
  %We show that $\{ \phib^{(l)} \}_{0 \leq l \leq n-1}$ can be taken to satisfy $(3)$ by replacing $\{ \phib^{(l)} \}$ suitably.
  %We set $\phib^{(n-1)} = \psib^{(n-1)}$.
  %Then $\phib^{(n-1)}$ satisfies the condition in (3) and the equality $\delta \phib^{(n-1)} + d''\psib^{(n-2)} = 0$.
  %We set $\zeta^{(n-2)} = \psib^{(n-2)}$.

  By $\HHH^{n}(M) = 0$, we take $\phi \in \Omega^{n-1}(M)$ such that $d\phi = \omega$.
  Let $\phib^{(n-1)}$ be the element corresponding to $\phi$ under the canonical isomorphism $\Omega^{n-1}(M) \cong \mathscr{A}^{0,n-1}(\MDfuldel_{\bullet})$.
  %Then $\phib^{(n-1)}$ satisfies the conditions in (2) and (3).
  By $d'' \delta \phib^{(n-1)} = - \delta d'' \phib^{(n-1)} = 0$ and $\HHH^{n-1}(M) = \HHH^{n-1}(M,\partial M) = 0$, we take an element $\psib^{(n-2)}$ of $\mathscr{A}^{1,n-2}(\MDfuldel_{\bullet})$ such that $\delta \phib^{(n-1)} + d'' \psib^{(n-2)} = 0$
  and that $\psib^{(n-2)}(g) \in \Omega^{n-2}(M, \partial M)$ for every $f \in \DM$.

  For $l < n-1$, we define $\phib^{(l)}$ and $\psib^{(l-1)}$ inductively as follows.
  Assume that $\phib^{(l+1)}$ satisfies the conditions in (2) and (3) and there exists $\psib^{(l)} \in \mathscr{A}^{n-l-1,l}(\MDfuldel_{\bullet})$ such that $\delta \phib^{(l+1)} + d''\psib^{(l)} = 0$ and that $\psib^{(l)}(f_1, \cdots, f_{n-l-1}) \in \Omega^{l}(M, \partial M)$ for every $(f_1, \cdots, f_{n-l-1}) \in \DM^{n-l-1}$.
  We fix $(g_1, \cdots, g_{n-l-1}) \in \DMful^{n-l-1}$, and take $(h_1 \cdots, h_{n-l-1}) \in \DMful^{n-l-1}$ satisfying $g_k|_{\partial M} = h_k|_{\partial M}$ for every $1 \leq k \leq n-l-1$.
  Then we have
  \begin{align*}
    &d''i^*\left(\psib^{(l)}(g_1, \cdots, g_{n-l-1}) - \psib^{(l)}(h_1, \cdots, h_{n-l-1})\right)\\
    =& -i^* \left(\delta\phib^{(l+1)}(g_1, \cdots, g_{n-l-1}) - \delta\phib^{(l+1)}(h_1, \cdots, h_{n-l-1})\right) = 0.
  \end{align*}

  If $l > 0$, then there exists $\zeta' \in \Omega^{l-1}(\partial M)$ such that
  \begin{align}\label{align:replacement}
    d''\zeta' = i^*\psib^{(l)}(g_1, \cdots, g_{n-l-1}) - i^*\psib^{(l)}(h_1, \cdots, h_{n-l-1})
  \end{align}
  since $\HHH^{l}(\partial M) = \HHH^{l}(S^{n-1}) = 0$.
  Since $i^* \colon \Omega^{l-1}(M) \to \Omega^{l-1}(\partial M)$ is surjective, we take $\zeta \in \Omega^{l-1}(M)$ such that $i^*\zeta = \zeta'$.
  %Here if $(g_1, \cdots, g_{n-l-1}) \in \DM^{n-l-1}$, then we take $0$ as $\eta'$ (and also as $\eta$).
  We set $\phib^{(l)}(h_1, \cdots, h_{n-l}) = \psib^{(l)}(h_1, \cdots, h_{n-l}) + d''\zeta$.
  Then $\phib^{(l)}$ satisfies the conditions in (2) and (3) by (\ref{align:replacement}).
  Since $d'' \delta \phib^{(l)} = - \delta d'' \phib^{(l)} = -\delta d''(\psib^{(l)} + d''\zeta) = \delta \delta \phib^{(l+1)} = 0$ and $\HHH^{l}(M) = \HHH^{l}(M,\partial M) = 0$, there exists $\psib^{(l-1)}$ such that $\delta \phib^{(l)} + d'' \psib^{(l-1)} = 0$
  and that $\psib^{(l-1)}(f_1, \cdots, f_{n-l}) \in \Omega^{l-1}(M, \partial M)$ for every $(f_1, \cdots, f_{n-l}) \in \DM^{n-l}$.

  If $l = 0$, then there exists a constant $a$ such that
  \begin{align*}
    a = i^*\psib^{(0)}(g_1, \cdots, g_{n-1}) - i^*\psib^{(0)}(h_1, \cdots, h_{n-1}).
  \end{align*}
  We set $\phib^{(0)}(h_1, \cdots, h_{n-1}) = \psib^{(0)}(h_1, \cdots, h_{n-1}) + a$.
  Then $\phib^{(0)}$ satisfies the conditions in (2) and (3).
  The resulting $\phib = \phib^{(0)} + \cdots +  \phib^{(n-1)}$ satisfies (1), (2) and (3).
\end{proof}

\begin{comment}
\begin{lemma}
  There exists $\phib = \phib^{(0)} + \cdots + \phib^{(n-1)} \in \mathscr{A}^{n-1}(\MDfuldel_{\bullet})$ such that $d \phib = \omb^{(n)} + \delta \phib^{(0)}$ and $\phib^{(l)}(g_1, \cdots, g_{n-l-1}) = \phib^{(l)}(h_1, \cdots, h_{n-l-1})$ if $g_i|_{\partial M} = h_i|_{\partial M}$ for every $1 \leq i \leq n-l-1$.
  %a sequence $\{ \phib^{(l)} \}_{0 \leq l \leq n-1}$ such that $\phib^{(l)} \in \mathscr{A}^{n-1-l, l}(\MDfuldel_{\bullet})$ for every $0 \leq l \leq n-1$ and
  %\begin{align*}
  %  d \left(\sum_l \phib^{(l)}\right) = \omb^{(n)} + \delta \phib^{(0)}.
  %\end{align*}
\end{lemma}
\end{comment}

Recall that $r \colon B\DMful^{\delta} \to B\Dbdr^{\delta}$ is the map induced by the boundary-restricting map.

\begin{lemma}\label{lem:delphi_descend}
  Let $\phib = \phib^{(0)} + \cdots + \phib^{(n-1)}$ be as in Lemma \ref{lem:phib_boundary_coincide}.
  Then there exists a unique closed $n$-form $\chi \in \mathscr{A}^{n}(B\Diff_0(\partial M)^{\delta}_{\bullet})$ such that $r^* \chi = \delta \phib^{(0)}$.
\end{lemma}

\begin{proof}
  Since $d'' \delta \phib^{(0)} = 0$, the element $\delta \phib^{(0)}(g_1, \cdots, g_n) \in \Omega^0(M)$ is a constant function for every $(g_1, \cdots, g_n) \in \DMful^{n}$.
  By (2) in Lemma \ref{lem:phib_boundary_coincide}, the constant $\delta \phib^{(0)}(g_1, \cdots, g_n)$ depends only on $(g_1|_{\partial M}, \cdots, g_n|_{\partial M})$.
  Hence, for every $(\gg_1, \cdots, \gg_n) \in \Diff_0(\partial M)^{n}$, we set
  \begin{align*}
    \chi(\gg_1, \cdots, \gg_n) = \delta \phib^{(0)}(g_1, \cdots, g_n),
  \end{align*}
  where $g_k$ is an element of $\DMful$ such that $g_{k}|_{\partial M} = \gg_k$ for every $1 \leq k \leq n$.
  It is easily verified that such $\chi$ is unique and a closed form.
\end{proof}

\begin{definition}\label{def:u}
  Let $\phib = \phib^{(0)} + \cdots + \phib^{(n-1)}$ be as in Lemma \ref{lem:phib_boundary_coincide}.  %Let us consider the wedge product $\phib^{(0)} \wedge \omb^{(n)}$ in $\mathscr{A}^{2n-1}(\MDfuldel_{\bullet})$ and its integration along the fiber
  We set
  \begin{align*}
    u = \int_M \phib^{(0)} \wedge \omb^{(n)} \in \mathscr{A}^{n-1}(\BDfuldelM_{\bullet}).
  \end{align*}
  %We set $u = \int_M \phib^{(0)} \wedge \omb^{(n)}$.
\end{definition}

%Let $\{ \phib^{(l)} \}_{0 \leq l \leq n-1}$ be as in Lemma \ref{lem:phib_boundary_coincide}.
%Let us consider the wedge product $\phib^{(0)} \wedge \omb^{(n)}$ in $\mathscr{A}^{2n-1}(\MDfuldel_{\bullet})$ and its integration along the fiber
%\begin{align*}
%  \int_M \phib^{(0)} \wedge \omb^{(n)} \in \mathscr{A}^{n-1}(\BDfuldelM_{\bullet}).
%\end{align*}
%We set $u = \int_M \phib^{(0)} \wedge \omb^{(n)}$.
%Let $j \colon B\DM^{\delta} \to B\DMful^{\delta}$ be the map induced by the inclusion.
For the fibration $B\DM^{\delta} \xrightarrow{j} B\DMful^{\delta} \xrightarrow{r} B\Dbdr^{\delta}$, we obtain the following transgression formula.
\begin{theorem}[transgression formula]\label{thm:transgression}
  Let $\phib = \phib^{(0)} + \cdots + \phib^{(n-1)} \in \mathscr{A}^{n-1}(\MDfuldel_{\bullet})$ be as in Lemma \ref{lem:phib_boundary_coincide}.
  Let $\chi \in \mathscr{A}^{n}(B\Diff_0(\partial M)^{\delta}_{\bullet})$ and $u \in \mathscr{A}^{n-1}(\BDfuldelM_{\bullet})$ be as in Lemma \ref{lem:delphi_descend} and Definition \ref{def:u}, respectively.
  Then the following hold:
  \begin{align*}
    [j^*u] = c(M) \ \text{ and } \  d u = \vol(M) \cdot r^* \chi.
  \end{align*}
  In particular, the class $c(M)$ transgresses to $\vol(M) \cdot [\chi]$.
\end{theorem}

%\begin{proposition}
%  The coboundary $\delta u$ coincides with $\vol(M)\cdot r^*\chi$, and the pullback $j^*u \in \mathscr{A}^{n-1}(\BDdelM_{\bullet})$ is a cocycle and represents the class $c(M)$.
%\end{proposition}

\begin{proof}
  Since $d u = \int_M \delta \phib^{(0)} \wedge \omb^{(n)}$ and $\delta \phib^{(0)} = r^* \chi$ by Lemma \ref{lem:delphi_descend}, we have
  \begin{align*}%\label{align:transgress_to_chi}
    d u =  \vol(M) \cdot r^*\chi.
  \end{align*}
  %and hence
  %\begin{align*}
  %  d (j^* u) = j^* d u = \vol(M) \cdot j^* r^* \chi = 0.
  %\end{align*}
  By Lemma \ref{lem:phib_boundary_coincide} (2) and the definition of $c(M)$, the form $j^* u$ represents $c(M)$.
\end{proof}

%According to the definition of the transgression by Borel (see Section \ref{sec:intro}), Proposition \ref{thm:transgression} is rephrased as follows.

%\begin{corollary}
%  For the exact sequence
%  \begin{align*}
%    1 \to \DM \to \DMful \to \Diff_0(\partial M) \to 1,
%  \end{align*}
%  the class $c(M)$ transgresses to the class $\vol(M) \cdot [\chi]$.
%\end{corollary}

%Note that the pullback of $u$ to $\BDdelM_{\bullet}$ represents the class $c(M) \in \HHH^{n-1}(\BDdelM_{\bullet})$. %, and the coboundary $\delta u$ is given by
%\begin{align*}
%  \delta u = \int_M \delta \phib^{(0)} \wedge \omb^{(n)} = \vol(M) \cdot \chi.
%\end{align*}
%This implies that $c(M)$ is contained in $E_{n}^{0,n-1} \subset \HHH^{n-1}(\BDfuldelM)$ and that $d_{n}^{0,n-1} c(M) = \vol(M)\cdot [\chi] \in E_{n}^{n,0}$.

Finally, we prove that the class $[\chi]$ coincides with $-\vol(M)$ times the Euler class of foliated $S^{n-1}$-bundles.
Let us consider the double complex $\mathscr{A}^{*, *}(S^{n-1} \DSdel_{\bullet})$ for the simplicial manifold $S^{n-1} \DSdel_{\bullet}$.
It is known that the spectral sequence $\{ E_r^{*,*}, d_r^{*,*} \}_{r\geq 1}$ of the double complex $\mathscr{A}^{*, *}(S^{n-1} \DSdel_{\bullet})$ is isomorphic to the Serre spectral sequence of
\begin{align}\label{univ_sphere_bundle}
  S^{n-1} \to E\DSdel \times_{\DSdel} S^{n-1} \to B\DSdel
\end{align}
from $E^2$-page (see \cite{losik93} for example).
Hence the Euler class $e \in \HHH^n(B\DSdel) \cong E_n^{n,0}$ of the bundle \eqref{univ_sphere_bundle} is given by $-d_n^{0,n-1}(\alpha)$, where $\alpha \in E_n^{0,n-1} \cong \HHH^{n-1}(S^{n-1})$ is the element corresponding to the generator of integral cohomology $\HHH^{n-1}(S^{n-1};\ZZ)$.

\begin{proposition}\label{prop:euler_cocycle}
  The class $[\chi] \in \HHH^{n-1}(\BDS)$ coincides with $-\vol(M) \cdot e$.
\end{proposition}

\begin{proof}
  We take $\omb^{(n)} \in \mathscr{A}^{0,n}(\MDfuldel_{\bullet})$ and $\phib = \phib^{(0)} + \cdots + \phib^{(n-1)} \in \mathscr{A}^{n-1}(\MDfuldel_{\bullet})$ as in Lemma \ref{lem:phib_boundary_coincide}.
  By Lemma \ref{lem:phib_boundary_coincide} (3), each pullback $i^*\phib^{(l)}$ to the boundary sphere descends to a form $\psib^{(l)} \in \mathscr{A}^{n-l-1, l}(S^{n-1}\DSdel_{\bullet})$.
  By Lemma \ref{lem:phib_boundary_coincide} (1), we have $d\psib = \delta \psib^{(0)}$.
  On the double complex $\mathscr{A}^{*,*}(S^{n-1}\DSdel_{\bullet})$, this is depicted as follows.
  \begin{align*}
  \xymatrix{
  0 &&&&& \\
  \psib^{(n-1)} \ar[u]_-{d''} \ar[r]^-{\delta} &0&&&& \\
  & \psib^{(n-2)} \ar[u]_-{d''} \ar[r]^-{\delta} &\ar@{.}[rrd]&&& \\
  &&&&& \\
  &&&& \psib^{(0)} \ar[u]_-{d''} \ar[r]^-{\delta} & \delta \psib^{(0)}
  }
  \end{align*}
  Note that $\delta \psib^{(0)} = \chi$ by the definition of $\chi$.
  By the definition of the spectral sequence of double complex, the zig-zag above implies that the form $\psib^{(n-1)}$ lives to $E_n^{0,n-1}$ and that the image $d_n^{0,n-1}([\psib^{(n-1)}])$ is equal to the class $[\chi] \in E_n^{n,0} = \HHH^n(\BDS)$ (see \cite[Section 14]{bott_tu82}).

  Recall that $\phib^{(n-1)} \in \mathscr{A}^{0,n-1}(\MDfuldel_{\bullet})$ is the element corresponding to $\phi \in \Omega^{n-1}(M)$ satisfying $d\phi = \omega$.
  By the Stokes formula, we have
  \begin{align*}
    \int_{S^{n-1}} i^* \phi = \int_M d\phi = \int_M \omega = \vol(M),
  \end{align*}
  which implies that $[\psib^{(n-1)}] = \vol(M) \cdot \alpha \in E_n^{0,n-1} \cong \HHH^{n-1}(S^{n-1})$.
  By the definition of the Euler class, we have
  \begin{align*}
    [\chi] = d_n^{0,n-1}([\psib^{(n-1)}]) = \vol(M)\cdot d_n^{0,n-1}(\alpha) = -\vol(M)\cdot e.
  \end{align*}
  This completes the proof.
\end{proof}

%More specifically from the assertion of Theorem \ref{thm:cM_transgression}, we obtain the following transgression formula:

\begin{proof}[Proof of Theorem \ref{thm:cM_transgression}]
  Theorem \ref{thm:transgression} and Proposition \ref{prop:euler_cocycle} imply the theorem.
\end{proof}

\section*{Acknowledgments}
The author would like to thank Professor Hitoshi Moriyoshi for helpful discussions and comments.
The author is partially supported by JSPS KAKENHI Grant Number JP23KJ1938 and JP23K12971.

\bibliographystyle{amsalpha}
\bibliography{references.bib}

\providecommand{\bysame}{\leavevmode\hbox to3em{\hrulefill}\thinspace}
\providecommand{\MR}{\relax\ifhmode\unskip\space\fi MR }
\providecommand{\MRhref}[2]{%
  \href{http://www.ams.org/mathscinet-getitem?mr=#1}{#2}
}
\providecommand{\href}[2]{#2}
\begin{thebibliography}{McD83}

\bibitem[Bor55]{MR72426}
Armand Borel, \emph{Topology of {L}ie groups and characteristic classes}, Bull.
  Amer. Math. Soc. \textbf{61} (1955), 397--432.

\bibitem[Bot78]{MR505649}
Raoul Bott, \emph{On some formulas for the characteristic classes of
  group-actions}, Differential topology, foliations and {G}elfand-{F}uks
  cohomology ({P}roc. {S}ympos., {P}ontif\'{\i}cia {U}niv. {C}at\'{o}lica,
  {R}io de {J}aneiro, 1976), Lecture Notes in Math., Vol. 652, Springer,
  Berlin-New York, 1978, pp.~25--61.

\bibitem[Bow11]{bowden11}
Jonathan Bowden, \emph{Flat structures on surface bundles}, Algebr. Geom.
  Topol. \textbf{11} (2011), no.~4, 2207--2235.

\bibitem[BT82]{bott_tu82}
Raoul Bott and Loring~W. Tu, \emph{Differential forms in algebraic topology},
  Graduate Texts in Mathematics, vol.~82, Springer-Verlag, New York-Berlin,
  1982.

\bibitem[Dup76]{MR413122}
Johan~L. Dupont, \emph{Simplicial de {R}ham cohomology and characteristic
  classes of flat bundles}, Topology \textbf{15} (1976), no.~3, 233--245.

\bibitem[Dup78]{D}
\bysame, \emph{Curvature and characteristic classes}, Lecture Notes in
  Mathematics, Vol. 640, Springer-Verlag, Berlin-New York, 1978.

\bibitem[HM94]{misc/24370500}
Steven Hurder and Yoshihiko Mitsumatsu, \emph{Transverse euler class of
  foliatious on non-atomic foliation cycles}, Contemporary Mathematics
  \textbf{161} (1994), no.~1994, 29--39.

\bibitem[Los93]{losik93}
Mark~V. Losik, \emph{Characteristic classes of transformation groups},
  Differential Geom. Appl. \textbf{3} (1993), no.~3, 205--218.

\bibitem[Mat79]{Mather2011}
John~N. Mather, \emph{On the homology of haefliger's classifying space},
  pp.~71--116, Springer Berlin Heidelberg, Berlin, Heidelberg, 1979.

\bibitem[McD83]{MR0678355}
Dusa McDuff, \emph{Some canonical cohomology classes on groups of volume
  preserving diffeomorphisms}, Trans. Amer. Math. Soc. \textbf{275} (1983),
  no.~1, 345--356.

\bibitem[Mit92]{MR1183731}
Yoshihiko Mitsumatsu, \emph{On the self-intersections of foliation cycles},
  Trans. Amer. Math. Soc. \textbf{334} (1992), no.~2, 851--860.

\bibitem[Mor16]{Moriyoshi16}
Hitoshi Moriyoshi, \emph{The {C}alabi invariant and central extensions of
  diffeomorphism groups}, Geometry and topology of manifolds, Springer Proc.
  Math. Stat., vol. 154, Springer, [Tokyo], 2016, pp.~283--297.

\bibitem[Tsu00]{tsuboi00}
Takashi Tsuboi, \emph{The {C}alabi invariant and the {E}uler class}, Trans.
  Amer. Math. Soc. \textbf{352} (2000), no.~2, 515--524.

\end{thebibliography}
\end{document}